\documentclass[11pt]{article}
\usepackage{latexsym, amssymb, amsmath, amsfonts, amsthm,graphics,
  combinedgraphics,euscript, enumitem,parskip}
\usepackage{caption}
\begin{document}

\theoremstyle{plain}

\newtheorem{prop}{Proposition}
\newtheorem{theorem}{Theorem}
\newtheorem{corollary}{Corollary}
\newtheorem{lemma}{Lemma}

\theoremstyle{definition}
\newtheorem{definition}{Definition}
\newtheorem*{remark}{Remark}
\newtheorem*{remarks}{Remarks}
\newtheorem{example}{Example}
\newtheorem{examples}{Examples}

\newcommand{\rr}{\mathbb R}

\newcommand{\nn}{\mathbb N}
\newcommand{\zz}{\mathbb Z}
\newcommand{\ts}{\thinspace}
\newcommand{\sse}{\subseteq}

\newcommand{\al}{\alpha}
\newcommand{\be}{\beta}
\newcommand{\ga}{\gamma}
\newcommand{\la}{\lambda}

\newcommand{\vv}{\mathit{v}}
\newcommand{\cc}{\mathcal{C}}
\newcommand{\ww}{\mathcal{W}}
\newcommand{\mcc}[1]{{\mathcal #1}}

\newcommand{\sen}{{\mathcal S}}
\newcommand{\rep}{{\mathcal R}}
\newcommand{\prs}{{\mathcal P}}
\newcommand{\vp}{{\mathcal V}}

\newcommand{\aji}{A_{j,i}}

\renewcommand\theenumi{(\alph{enumi})}
\renewcommand\labelenumi{\theenumi}

\newcommand{\vsp}{\vspace{.2cm}}

\title{Power Index Rankings in  Bicameral Legislatures and the US Legislative System}

\author{Victoria Powers \footnote{
Department of Mathematics,
Emory University,
Atlanta, GA 30322. Email:\, vpowers@emory.edu}}
\maketitle

\begin{abstract}  In this paper we study rankings induced by power indices of players in simple game models of bicameral legislatures.  For a bicameral legislature where bills are passed with a simple majority vote in each house we give a condition involving
the size of each chamber which guarantees  that a member of the smaller house has more power than a member of the larger house, regardless of the power index used.  The only case for which this does not apply is  when the smaller house has an odd number of players, the larger house has an even number of players, and the larger house is less than
twice the size of the smaller house.   We explore what can happen in this exceptional case.  These results generalize to multi-cameral legislatures.  
Using a standard model of the US legislative system as a simple game, we use our results to study power index rankings of the four types
of players -- the president, the vice president, senators, and representatives. We prove
that a senator is always ranked above a representative and ranked the same as  or above the vice president.  We also show that
the president is always ranked above 
the other players.  We show that for most power index rankings, 
including the Banzhaf and Shapley-Shubik power indices, the vice president is ranked above a representative, however,
there exist power indices ranking a representative above the vice president.  
\end{abstract}

\section{Introduction}   

A power index assigns a numerical measure of power to each player in a simple game and thus yields a ranking of the players. In this paper we look at power index rankings of the
players in simple game models of bicameral legislatures and similar legislative systems.   For a bicameral legislature where bills are passed with a simple majority vote in each house we give a condition involving
the size of each chamber which guarantees  that a member of the smaller house has more power than a member of the larger house, regardless of the power index used.  The only case for which this does not apply is  when the smaller house has an odd number of players, the larger house has an even number of players, and the larger house is less than
twice the size of the smaller house.   We explore what can happen in this exceptional case.   These results generalize
easily to the multi-cameral situation.

We apply our results and techniques to study power index rankings in the standard simple game
model of the US legislative system, which has four types of players:  senators, representatives, the president, and the vice president.  In the case of the US legislative system, we show that regardless of the power index chosen,  the president always has more power than the other players,
a senator always has more power than a representative, and a senator always has at least as much power as the vice president.  For ``reasonable" power indices,
including the Banzhaf
and Shapley-Shubik indices, the ranking  from most power to least power is:  president, senator, vice president, representative.   These results apply to more
general systems which are similar to the US legislative system, with a bicameral legislature plus a president  or a bicameral legislature plus a president and vice president.

Power indices for this simple game model of the US legislative system have been studied previously, in the context of calculating power with a specific index.   For example, in the book
by Taylor and Pacelli \cite{tp}, the authors discuss calculating Banzhaf and Shapley-Shubik power in a model of the US legislative system (without the vice president).  
Brams, Affuso, and Kilgour \cite{bak} study Banzhaf and Johnston power in the same model (no vice president).  In both of these cases, the authors are interested
in the percent of power held by the players rather than  power index rankings.

\section*{Acknowledgements}  Thanks to Bruce Reznick for several helpful conversations and to several anonymous referees 
who provided extensive comments and suggestions
that  greatly improved the paper.  

\section{Preliminaries}

\subsection{Simple Games and Power Indices}

A  (monotonic) {\bf simple game} is a pair $(N, \ww)$ where $N = \{ 1, 2, \dots, n \}$ is the set of players and $\ww$ is a set
of subsets of $N$, called the {\bf winning coalitions}, such that 
\begin{itemize}
\item $\emptyset \not \in \ww$.

\item $N \in \ww$.

\item If $S \in \ww$ and $S \sse T$, then $T \in \ww$.

\end{itemize}

The  {\bf minimal winning coalitions} are the
winning coalitions for which no proper subset is winning.  The set of winning coalitions
is determined by the minimal ones since a subset $S \in \ww$ if and only if $S$ contains a minimal winning coalition.

A simple game is a model of a yes-no voting system in which the players are deciding on a single alternative such as a motion, bill, or amendment. 
The winning coalitions are precisely the sets of players that can force a bill to pass if they all support it.

Given a simple game $(N,\ww)$ and $S \in \ww$ containing player $i$, we say $i$ is {\bf critical} in
$S$ if $S$ is winning and $S \setminus \{ i \}$ is losing.  For $i \in N$ and $1 \leq k \leq n$, 
let $$\cc_i = \{ S \in \ww \mid i \text{ is critical in } S  \}, \quad  \cc_i(k) = \{ S \in \cc_i \mid |S| = k \}.$$
Let $c_i(k) = |\cc_i(k)|$, the number of coalitions of size $k$ in which $i$ is critical.  The numbers $c_i(k)$ are
called {\bf critical numbers}. \\

\begin{definition} Define a binary relation on $N$ by $i \succeq j$ iff $c_i(k) \geq c_j(k)$ for all $k$ such that $1 \leq k \leq n$, and write  $i \succ j$ if  $c_i(k) >  c_j(k)$  for all $k$
such that $c_i(k)$ and $c_j(k)$ are not both zero.  Following 
\cite{CF}, we call $\succeq$  the {\bf weak desirability} relation.
\end{definition}

\subsection{Power Indices}

Power indices are a way to measure the relative power of the players in a simple game.  The most famous of these are
the Shapley-Shubik index \cite{SS} and the Banzhaf index \cite{banz}.  Semivalues were introduced in 1979 by Weber \cite{web}
as a generalization of the notion of a power index to general cooperative games.   Dubey et al.~\cite{dub} show that semivalues can be characterized in terms of a 
{\bf weighting vector}
$(\la_1, \dots , \la_n)$ such that $\la_k \geq 0$ for all $k$ and $\sum_{k=1}^n \la_k \binom{n - 1}{k - 1} = 1$.   

Given a power index $\Phi$ with weighting vector $(\la_1, \dots , \la_n)$, the $\Phi$-power of a player $i$ is defined by
$$\Phi(i) := \sum_{k=1}^n \la_k \thinspace c_i(k).$$    Thus the $\la_i$'s give a weighting of 
a player's contribution to coalitions of size $k$.

The Shapley-Shubik power index is defined by weighting coefficients $\la_k = 1/\left(n \binom{n-1}{k-1}\right)$ and the Banzhaf power index 
is defined by weighting coefficients $\la_k = 1/2^{n-1}$.

Any power index $\Phi$ defines a ranking on the set of players in a simple game and we write  $i \geq_\Phi j$ to denote that
$\Phi(i) \geq \Phi(j)$ and $i >_\Phi j$ if $\Phi(i) > \Phi(j)$.   Clearly, different power indices can lead to different rankings for the same game.   In \cite{SK}, Saari and Sieberg
look at rankings of players coming from power indices in cooperative games.  They show that different indices can generate radically
different rankings and that there can be many different rankings even for games with a relatively small number of players.  This is in contrast to our results
which show that there are only two possible rankings for a simple game model of the US legislative system.  

The following results follow easily from the definitions, see also \cite[Theorem 3.4]{CF}.  \vsp

\begin{prop}  \label{compare} Let $i$ and $j$ be two players in a simple game.

\begin{enumerate}

\item If $i \succeq j$, then for any power index $\Phi$, $i \geq_\Phi j$. 

\item Suppose there exist $k, m$ such that $c_i(k) > c_j(k)$ and $c_i(m) < c_j(m)$.  Then there exist power indices $\Phi$ and $\Psi$ such that
$i >_\Phi j$ and $j <_\Psi i$.

\end{enumerate}

\end{prop}

\section{Bicameral Voting Systems}

We look at power in a bicameral legislative system where bills are passed with a simple majority in each house.  We show that a member of the smaller house  has more power than a member of the larger house regardless of 
the choice of power index used to measure power, apart from the following case:  The smaller house has an odd
number of players, the larger house has an even number of players, and the larger house is less than twice the size of the smaller house.  
In this exceptional case, the choice of power index will determine whether the members of the smaller house or the larger house
have the most power.

Recall that for $n, k \in \nn$ with $0 \leq k \leq n$,
the {\bf binomial coefficient} $\binom{n}{k}$ denotes the number of ways of choosing $k$ elements from a set of $n$ elements. 
For the simple games we study in this work, formulas for the critical numbers involve products of binomial coefficients.  Results 
on binomial coefficients needed in this section can be found in the appendix.

Suppose we have a simple game with two types of players, say that there are $m_s$ senators and $m_r$ representatives, and we
want to count the number of coalitions that contain a specific senator and have exactly $x$ senators and $y$ representatives.  
Creating such a coalition
consists of choosing $x - 1$ players from $m_s - 1$ senators and choosing $y$ players from the $m_r$ representatives.  Thus there are
$\binom{m_s - 1}{x - 1} \cdot \binom{m_r}{y}$ such coalitions.   

\vsp
Let $c_s(k)$ denote the number of coalitions of size $k$ in which a senator is critical and define $c_r(k)$ similarly
for a representative. 

 \vsp

\begin{lemma}  \label{lem1} 
Suppose that  $q_s$ and $q_r$ are the the quotas for passage of a bill, i.e., the minimal winning coalitions
consist of $q_s$ senators and $q_r$ representatives.  Then

\begin{enumerate}[label=(\roman*)]
\item $c_s(k) \neq 0$ iff $q_s + q_r  \leq k \leq q_s + m_r$ and $c_r(k) \neq 0$ iff $ q_s + q_r  \leq k \leq q_r + m_s$.

\item  For $k \in \nn$ with $q_s + q_r  \leq k \leq q_s + m_r$,  $$c_s(k) = \binom{m_s - 1}{q_s - 1} \cdot \binom{m_r}{k - q_s}.$$

\item  For $k \in \nn$ with $q_s + q_r \leq k \leq q_r + m_s$,  $$c_r(k) = \binom{m_r - 1}{q_r - 1} \cdot \binom{m_s}{k - q_r}.$$

\end{enumerate}

\end{lemma}

\begin{proof} Since the minimal winning coalitions are exactly the coalitions with $q_s$ senators and $q_r$ representatives, the coalitions in which a fixed senator is critical consist of the senator plus $q_s - 1$ of the $m_s - 1$ other senators along with between $q_r$ and $m_r$ 
of the $m_r$ representatives.  The coalitions in which 
a particular representative is critical 
consist of the representative plus $q_r - 1$ of the $m_r - 1$ other representatives along with between 
$q_s$ and $m_s$ senators.  The assertions follow easily from these observations.
\end{proof}

The lemma implies that to show $c_s(k) > c_r(k)$ we must prove the following inequality:
\begin{equation}  \label{ineq1}
\binom{m_s - 1}{q_s - 1} \cdot \binom{m_r}{k - q_s} >   \binom{m_r-1}{q_r-1} \cdot \binom{m_s}{k - q_r}     \end{equation}    This inequality  is studied in the appendix.  \vsp

For the rest of this section we assume that $m_s < m_r$  and the quotas correspond to simple majorities,
 so that $q_s = \lceil (m_s + 1)/2 \rceil$ and  $q_r = \lceil (m_r +1)/2 \rceil$.   \vsp

\begin{lemma}  \label{lem1.5}    We have $m_r -q_r \geq m_s - q_s$.  Hence, by Lemma \ref{lem1}, $c_s(k) = 0$ implies $c_r(k) = 0$.
\end{lemma}

\begin{proof}  Suppose $m_r$ is odd, say $m_r = 2x + 1$, then $q_r = x +1$.  If $m_s = 2y +1$, then $q_s = y +1$ and $x > y$, hence 
$m_r - q_r = x  >  y = m_s - q_s$.  If $m_s = 2y$, then
$q_s = y + 1$ and $x \geq y$, hence $m_r - q_r = x >  y - 1 = m_s - q_s$.   Now suppose $m_r$ is even, say $m_r = 2x$, then $q_r = x + 1$.  If $m_s = 2y + 1$, then $x > y$ and
$m_r-q_r = x- 1 \geq y = m_s - q_s$.  If $m_s = 2y$, then $x > y$ and $m_r - q_r = x- 1 > y - 1 = m_s - q_s$.  
\end{proof}

\vsp

  \begin{prop} \label{cor1}   If  $q_s \cdot m_r > q_r \cdot m_s$, then
  $c_s(k) > c_r(k)$ for all $k$ such that $c_s(k) \neq 0$.

    \end{prop}
    
    \begin{proof}   We need to show that  Inequality \eqref{ineq1} holds for all $k$ with $q_r + q_s \leq k \leq m_r + q_s$.

    Case 1:  If 
    $q_r + q_s  \leq k \leq q_r + m_s$, then $c_r(k) \neq 0$ and we use Proposition \ref{main} (d).  
    Since $q_s \cdot m_r > q_r \cdot m_s$, we need only show that 
    $(m_r - q_r)(q_s + 1) > (m_s - q_s)(q_r + 1).$  
    This is equivalent to $m_r q_s + m_r - q_r > m_s q_r + m_s - q_s$.   Since $m_r \cdot q_s > m_s \cdot q_r$ by assumption and
    $m_r - q_r \geq m_s  -q_s$ by Lemma \ref{lem1.5}, the inequality holds.  \vsp
    
    Case 2:  If $m_s + q_r < k \leq m_r + q_s$, then $c_r(k) = 0$ and $c_s(k) \neq 0$, by Lemma \ref{lem1} (a).  Thus
    $c_s(k) > c_r(k)$ is clear.
    
    \end{proof}
    
    Note that $q_s \cdot m_r > q_r \cdot m_s$ is equivalent to $q_s/m_s > q_r/m_r$ and thus   the proposition says that as long 
    $m_s < m_r$ and the proportion of the smaller house needed to pass a bill is larger than the proportion needed
    in the bigger house, then for any senator $\sen$ and any representative $\rep$, $\sen \succ \rep$.

\vsp

\begin{theorem}  \label{bicam_thm}  If $m_s < m_r$ and the quotas $q_s$ and $q_r$ correspond to simple majorities, 
then in each of the following cases, $c_s(k) > c_r(k)$ for all $k$ such that $c_s(k) \neq 0$:  
\begin{enumerate}[label=(\roman*)]

\item $m_s$ and $m_r$ are both odd,
\item $m_s$ and $m_r$ are both even, 
\item $m_s$ is even and $m_r$ is odd,
\item $m_s$ is odd, $m_r$ is even, and $m_r > 2m_s$.  

\end{enumerate}

Thus in these cases, for a senator $\sen$ and a representative $\rep$, $\sen >_\Phi \rep$  for any power index $\Phi$ that does
not assign $0$ power to both.

\end{theorem}

\begin{proof}  
By Proposition \ref{cor1}, we need only show that $ q_s \cdot m_r > q_r \cdot m_s$, equivalently,
$ q_s \cdot m_r - q_r \cdot m_s > 0$.
\vsp

Case (i).  If $m_s$ and $m_r$ are both odd, say $m_s = 2x + 1$ and $m_r = 2y + 1$ with $x < y$,  then $q_s = x + 1$, $q_r = y + 1$ 
and $m_s \cdot q_r - m_r \cdot q_s = y - x > 0$.
 \vsp 

Cases (ii) and (ii) are similarly 
easy to check.
\vsp

 Case(iv).  Suppose $m_s = 2 x + 1$,  $m_r = 2y$,  and $m_r > 2 m_s$.    Then 
 $q_s \cdot m_r - q_r \cdot m_s = y - 2x - 1$, hence $q_s \cdot m_r - q_r \cdot m_s > 0$ iff $y > 2x + 1 = m_s$ iff $2y > 
 4x + 2$, i.e., iff $m_r > 2 m_s$.   \vsp

\end{proof}

\subsection{The Exceptional Case}

We now consider the remaining case not covered by Theorem \ref{bicam_thm}:  $m_s$ is odd, $m_r$ is even, and $m_r \leq 2m_s$.    In this case, the relationship between the 
$c_s(k)$'s  and $c_r(k)$'s  is more complicated.   Assume $m_s = 2x + 1$ and $m_r = 2y$, so that $q_s =  x + 1$ and $q_r = y + 1$.  We first show that in this case, for the minimal
winning coalitions ($k = q_s + q_r$) the critical number for a member of the larger house  is greater than or equal to the critical number
for a member of the smaller house.    \vsp

\begin{lemma}  \label{excep}  With the above assumptions, $c_s(q_s + q_r) < c_r(q_s + q_r)$ if $m_r < 2m_s$.  If $m_r = 2m_s $, then $c_s(q_s + q_r)
= c_r(q_s + q_r)$ .  

\end{lemma}

\begin{proof}  By Lemma \ref{mincase}, $c_s(q_s + q_r) <  c_r(q_s + q_r)$ if $q_s m_r  < q_r m_s$ and $c_s(q_s + q_r) = c_r(q_s + q_r)$ if $q_s m_r  = q_r m_s$.  
If $m_r \leq  2m_s$, then  $ y < 2x + 1$.  We have 
$q_s m_r  - q_r m_s  = 2y(x + 1)  - (2x + 1)(y + 1) = y - (2x +1)$.  If $m_r < 2m_s$, then $y < 2x + 1$ and $q_s m_r  - q_r m_s <0$, 
hence $q_s m_r  < q_r m_s$.
If $m_r = 2m_s$, then $y = 2x + 1$ and $q_s m_r  < q_r m_s$.   
\end{proof}

The relationship in the case where  $m_r = 2m_s$, i.e., the gap between the sizes of the two houses is as large
as possible, is almost the same as in
the non-exceptional cases:  \vsp

\begin{prop}   Suppose $m_s = 2x + 1$ and $m_r = 2m_s = 4x + 2$.  Then 

\begin{enumerate}

\item $c_s(k) \neq 0$ for $3x + 3 \leq k \leq 5x + 3$.

\item $c_s(3x+3)  = c_r(3x+3)$

\item $c_s(k) > c_r(k)$  for $3x+4 \leq k \leq 5x + 3$.

\end{enumerate}
It follows that $\sen >_\Phi \rep$ for any power index apart from the one that assigns weight $\la_k = 0$ for all $k \neq 3x + 3$, which ranks $\sen$ and $\rep$ equally.  

\end{prop}

\begin{proof}   We have $q_s = x + 1$ and $q_r = 2x + 2$.

(a) follows from Lemma \ref{lem1}.

(b)   follows from Lemma \ref{excep}.

(c) We need to prove that inequality \eqref{ineq1} holds for all $k$ in the given range.    By Corollary \ref{cor2} in the appendix, it is enough to prove it for $k = q_r + q_s + 1 = 3x + 4$.  
By Proposition \ref{main} (a) and using the fact that 
$m_r = 2m_s$ and $q_r = 2q_s$, the inequality holds iff  $$\frac{4m_s(m_s-q_s)}{2q_s(2q_s +1)} > \frac{m_s(m_s-q_s)}{q_s(q_s+1)}, $$
which is equivalent to 
$2q_s +2 > 2q_s + 1$.   

\end{proof}

We now look at the remaining ``extreme" case  within the exceptional case, i.e., the case where the gap between the two houses is as
small as possible.
 \vsp

\begin{prop}  Suppose  $m_s = 2x + 1$ and $m_r = 2x + 2$ so that $q_s = x + 1$ and $q_r = x + 2$.  Then

\begin{enumerate}
\item    $c_s(k) \neq 0$ for
$2x + 3 \leq k \leq 3x + 3$ and $c_r(k) \neq 0$ iff $c_s(k) \neq 0$ 

\item  $c_s(k) < c_r(k)$ for all
$k$ such that $c_s(k) \neq 0$.  

\end{enumerate}
It follows that for any representative $\rep$ and any senator $\sen$, $\rep \succ \sen$ in this case.

\end{prop}

\begin{proof} 
(a)  The minimal winning coalitions have size $q_s + q_r = 2x + 3$ and the largest coalition for which $\sen$ or $\rep$
is critical has size $q_s + m_r = 3x + 3 = q_r + m_s$.   

(b) $c_s(2x+3) < c_r(2x+3)$ by Lemma \ref{excep}.   By Corollary \ref{cor2} in the appendix, if
$c_s(k) > c_r(k)$ for some $k > 2x +3$, then $c_s(i) > c_r(i)$ for all $i \geq k$.  Thus it is enough to
show that this fails for the largest possible value of $k$, i.e.,  it is enough to show that
$c_s(3x + 3) < c_r(3x + 3)$, which is easy:
$$c_s(3x+3) = \binom{2x}{x} \cdot \binom{2x+2}{2x+2} < \binom{2x + 1}{x + 1} \cdot \binom{2x+1}{2x+1} = c_r(3x+3).$$

\end{proof}

Here is a summary of what happens in this exceptional case.  
Fix $m_r$ and $m_s$ and assume $m_r \neq 2m_s$.  As noted above,  we have $q_s \cdot m_r > q_r \cdot m_s$ so that the proportion of the 
smaller house needed for passage is less than the proportion that
the bigger house needs.  This shifts the advantage to the bigger house for the minimal winning coalitions.   Then in the first extreme case,  i.e., 
when the gap between the sizes of the
two houses is as large as possible, the advantage only helps for minimal winning coalitions, so that for all other sizes of
coalitions, the smaller house has the advantage.  
  In the second extreme case, i.e.,   
  when the gap between the size of the two houses is as small as possible, the advantage stays with the bigger house for all 
  coalition sizes.   
   \vsp

In cases in between the two extremes, what happens is that the bigger house starts out with an advantage, then
at some point the advantage shifts to the smaller house and remains there as the sizes of the coalitions increase.  For the
biggest possible gap, this shift occurs as soon as the coalitions are no longer minimal.   For the first extreme case, this
shift never happens.   The bigger the gap, the sooner the shift will happen.
As an example, consider a case in between the two extremes:  $m_s = 101$ and $m_r = 150$, so that $q_s = 51$ and $q_r = 76$.
Then we have $c_s(k) < c_r(k)$ for the two smallest values of $k$ ($k = 127$ and $k = 128$)  $c_s(k) > c_r(k)$ for the remaining 
values of $k$ for which $c_s(k) \neq 0$.    \vsp

What this means is that in the first extreme case (gap between the size of the houses is as large as possible),  for all 
power indices $\phi$, members of the smaller house have more power than members of the larger house.  For 
the second extreme case (gap between the size of the houses is as small as possible), the members of the larger
house have more power than the members of the smaller house for any power index.   Finally, for the case in between
the two extremes, there will be power indices for which members of the smaller house have more power and others
for which member of the larger house have more power.

% If $n$ is odd and $m$ is even, then $c_s(k) > c_r(k)$ for all $k$ such $c_s(k) \neq 0$ if and only if
%$m > 2 n$.
%It follows that if $m \leq 2q$ then $c_s(k) > c_r(k)$  fails for $k = a + b$.  

\subsection{Generalization to Multi-cameral Legislatures}

The results on bicameral legislatures generalize relatively easily to a legislative body with any number of
houses.  Assume that we have $n$ houses, denoted $\mcc{H}_1, \dots , \mcc{H}_n$ and that $\mcc{H}_j$ has
$m_j$ members.     The relationship between the power of a member of $\mcc{H}_j$ and $\mcc{H}_l$ is the same as
the relationship they would have if there were only two houses.   Intuitively, this makes sense because being critical in a coalition for
a particular member of a house is independent of the makeup of the coalitions in other houses.  \vsp

For each $j$, let $q_j$ denote the minimum number of votes needed in $\mcc{H}_j$ to pass a motion.  
The following notation will be useful.  Let $[n]$ denote $\{ 1, 2, \dots , n \}$.  Given $I \sse [n]$ and $k \in \nn$, let $U_I(k)$ denote the number
of different subsets $S \sse \cup_{j \in I} \mcc{H}_j$ of size $k$ such that for each $j \in I$, $\lvert S \cap \mcc{H}_j \rvert \geq q_j$.
In other words, $U_I(k)$ is the number of ways of building a set of coalitions, one from each house in $\{ \mcc{H}_j \}_{j \in I}$,
such that each coalition in $\mcc{H}_j$ meets the threshold $q_j$ needed to pass legislation.  Then
$0 < U_I(k) \leq \sum_{j \in I} m_j$.
Using this notation, notice that for a fixed $j$, a specific member $r \in \mcc{H}_j$,  and $k \in \nn$ we have
$$c_r(k) = \binom{m_j - 1}{q_j - 1} \cdot U_{[n] \setminus \{ j \}} (k - q_j).$$

\begin{theorem} \label{multicam_thm}  With notation as above, suppose that each $q_j$ represents a simple majority.
 Let $r \in \mcc{H}_j$ and $s \in \mcc{H}_l$, where $j \neq l$.   Then 

\begin{enumerate}
\item   For each of the following cases, $c_s(k) > c_r(k)$ for all $k \in \nn$ such that $c_s(k) \neq 0$:  $m_j < m_l$ and $m_j$ and $m_l$ are both odd, both even,
$m_j$ is even and $m_l$ is odd, or $m_j$ is odd, $m_l$ is even and $m_l > 2 m_j$.

\item  If $m_j$ is odd and $m_l$ is even, then the relationship between $c_s(k)$ and $c_r(k)$ mirrors the relationship detailed in
Section 3.1.
\end{enumerate}
\end{theorem}

\begin{proof}  We can build all coalitions of size $k$  in which $r$ is critical as follows: Choose $q_j - 1$ of the $m_j - 1$ members
of  $\mcc{H}_j$ that are not $r$, then choose $q_l + d$ members of $\mcc{H}_l$, where $d$ ranges from $0$ to $m_l - q_l$, and 
finally choose, if possible,  subsets of the remaining houses which meet the minimum so that the size of the resulting winning coalition is $k$.   Thus 
 $$c_r(k) =    \sum_{d=0}^{m_l - q_l}  \binom{m_j - 1}{q_j - 1} \cdot  \binom{m_l}{q_l + d} \cdot 
U_{[n] - \{ i,j \}} (k - q_j - q_l - d),$$ and similarly, 
$$ c_s(k) =   \sum_{d=0}^{m_j - q_j}\binom{m_l - 1}{q_l - 1} \cdot  \binom{m_j}{q_j + d} \cdot 
U_{[n] - \{ i,j \}} (k - q_j - q_l - d).$$   

(a)  It is easy to check in this case that $m_l - q_l > m_j - q_j$ and thus, using the formulas above, to show $c_r(k) > c_s(k)$ it
is enough to show that for $0 \leq d \leq m_j - q_j$, $\binom{m_j - 1}{q_j - 1} \cdot  \binom{m_l}{q_l + d}  >
\binom{m_l - 1}{q_l - 1} \cdot  \binom{m_j}{q_j + d}$.  This follows from Proposition \ref{main}  exactly as in the proof
of Theorem \ref{bicam_thm}.

(b)  The same argument works in this case.

\end{proof}

\section{The US Legislative System}

We apply our results on bicameral legislative systems to study power in the US legislative system.  
We model the US legislative system as a simple game with 537 players:  
the president, vice president, 100 senators in the Senate, and 435 representatives in  the House of Representatives.
A bill passes if a majority of the senators and a majority of the representatives
vote yes and the president  signs the bill.  If the president does not sign the bill, it can be passed with a supermajority of at least 67 
senators and 290 representatives.   The role of the vice president is to break ties in the Senate.  For winning coalitions in which the president 
is critical, the vice president plays the same role as a senator; in these cases we can assume
that the senate contains 101 players.  We will call the set of senators plus the vice president the ``full Senate".

There are two types of minimal winning coalitions:  

\begin{itemize}
\item[I.] 51 from the full Senate, 218 
representatives, and the president;

\item[II.]  67 senators and 290 representatives.

\end{itemize}

We look at critical instances for the four types of players in order to compare the critical numbers.   Note that if a winning coalition contains exactly $51$ senators, then every senator is critical and adding the
vice president yields a coalition in which no senator is critical.  Apart from this case, if a player who is not the vice president is critical in a coalition that
does not contain the vice president, then this player is still critical if the vice president is added to the coalition.   

For ease of exposition, we write $\prs$ and $\vp$ for the president and vice-president and let $\sen$ be a fixed senator and
$\rep$ a fixed representative.  
We write $c_p(k)$ (resp. $c_v(k)$, $c_s(k)$, $c_r(k)$) for the number of coalitions of size $k$ in which the president
(resp. the vice-president, a senator, a representative) is critical.  The following table lists the different types of coalitions, along with their sizes, in which
 the president (P1 - P4), a senator (S1 - S3),  a  representative (R1 - R3), or the vice president (V) 
 are critical, along with the possible sizes.  \\[.2cm]

\captionof{table}{Critical numbers in the US system} \label{tab1}
\begin{tabular} {| c | c | c |}
\hline
Type & Members & Size \\
\hline
P1 & $\prs$, $51 - 66$ from the Senate, $218 - 435$ representatives & $270-502$ \\
\hline
P2 & $\prs$, $\vp$, $50 - 66$ from the Senate , $218 - 435$ representatives & $270-503$ \\
\hline
P3 & $\prs$,  $67 - 100$ from the Senate, $218 - 289$ representatives & $286-390$ \\
\hline
P4 & $\prs$,  $\vp$, $67 - 100$ from the Senate, $218 - 289$ representatives & $287-391$ \\
\hline
S1 & $\prs$, $\sen$, 50 others from the full Senate , 218-435 representatives & $270 - 487$ \\
\hline
S2 &  $\sen$, $66$ other senators, $290 - 435$ representatives & $357 - 502$ \\
\hline
S3 &  $\sen$, $66$ other senators, $290 - 435$ representatives, $\vp$ & $358 - 503$ \\
\hline
R1 & $\prs$, $\rep$, $217$ other representatives, $51 - 101$ from the full Senate & $270 - 320$ \\
\hline
R2 & $\rep$, $289$ other representatives, $67$ senators & $357$ \\
\hline
R3 & $\rep$, $289$ other representatives, $68 - 101$ from the full Senate & $358 - 391$ \\
\hline
V & $\vp$, $\prs$, $50$ from the Senate, and $218-435$ representatives & $270 - 487$ \\ 
\hline
\end{tabular}

\vsp \vsp

\begin{prop}  \label{p}  

\begin{enumerate}
\item If $c_p(k) = 0$, then $c_v(k) = c_r(k) = c_s(k) = 0$.

\item For all $k \in \nn$ such that  $c_p(k) \neq 0$,  $c_p(k) > c_v(k)$
\item  For all $k \in \nn$ such that  $c_p(k) \neq 0$, $c_p(k) > c_s(k)$.
\end{enumerate}

\end{prop}

\begin{proof}  (a) follows immediately from Table 1, the table of coalition sizes.  \vsp

Fix $k$ such that $c_p(k) \neq 0$, then $270 \leq k \leq 503$.    Let $\cc_p(k)$ denote the set of coalitions of size $k$ in which the
president is critical and $\cc_s(k)$ the set of coalitions of size $k$ in which a senator is critical.  \vsp

(b): Every coalition in which $\vp$ is critical contains $\prs$, and $\prs$ is also critical, thus $c_p(k) \geq c_v(k)$.   In addition, given $S \in \cc_v(k)$, the coalition
formed by removing $\vp$ and adding a senator not already in $S$ is in $\cc_p(k)$ and not in $\cc_v(k)$.  Hence $c_p(k) > c_v(k)$.
\vsp
 
(c): Define a function
$f : \cc_s(k) \to \cc_p(k)$ as follows:  Given $S$ in $\cc_s(k)$, if $S$ is type S1, then $\prs$ is critical in $S$ and we define $f(S) = S$.  If $S$ is type S2 or S3, then the coalition 
$S' =  (S \setminus \{ \sen \}) \cup \{ \prs \}$ is  in $\cc_p(k)$ since it contains only 66 senators, and we define 
$f(S) = S'$.  Then $f$ is clearly injective, hence $c_p(k) \geq c_s(k)$.  To show that the inequality is strict we need only show that
$f$ is not surjective.  \vsp

If $270 \leq k \leq 356$, then there are no type S2 or S3 coalitions in $\cc_s(k)$.  Thus any coalition in $\cc_p(k)$ that does not contain $\sen$ is not in Im $ f$, and there are clearly many
of these.  Now suppose $357 \leq k \leq 502$ and let $\tilde S$ be a coalition in $\text{Im } f$ of the form $(S \setminus \{\sen\}) \cup \{\prs\}$ with $S \in \cc_s(k)$ of type S2 or S3.  Then $\tilde S$ has exactly 66 senators and 434 or less representatives and we can construct
a new coalition in $\cc_p(k)$ by replacing any senator in $\tilde S$ by a representative not already in $\tilde S$.  This clearly yields a 
coalition in $\cc_p(k)$ that is not in $\text{Im } f$.  Hence $f$ is not surjective in this case.  \vsp

For $k = 503$, coalitions in $\cc_s(k)$ consist of $\sen$ plus $66$ other senators, 435 representative, and $\vp$; while those in $\cc_p(k)$ consist of $\prs$ plus $66$ senators, 435 representatives,
and $\vp$.  Then $c_p(503) = \binom{100}{66} > \binom{99}{66} = c_s(503)$.  Therefore in all cases we have $c_p(k) > c_s(k)$.

\end{proof}

\begin{prop} \label{sv} If $c_s(k) = 0$, then $c_v(k) = 0$.   If $c_s(k) \neq 0$, then  $270 \leq k \leq 503$ and we have
\begin{enumerate}[label=(\roman*)]

\item If $270 \leq k \leq 356$, then $c_s(k) = c_v(k)$.  

\item If $357 \leq k \leq 503$, then $c_s(k) > c_v(k)$.
\end{enumerate}

\end{prop}
 
\begin{proof}  The first statement follows immediately from Table 1.  Fix $k$ with $270 \leq k \leq 503$.  Coalitions in $\cc_v(k)$ consist of $\vp$ plus $50$ senators, $k-52$ representatives, and the president, hence
for $270 \leq k \leq 487$, 
$$c_v(k) = \binom{100}{50} \cdot \binom{435}{k - 52} .$$

(i)  Since $k  \leq 356$, coalitions in $\cc_s(k)$ are type S1 only, thus they 
consist of $\sen$ plus  $50$ others from the full Senate, $k - 52$ representatives, and $\prs$.  Hence
$$c_s(k) = 
\binom{100}{50} \cdot \binom{435}{k - 52} = c_v(k).$$

\noindent
(ii)  For $488 \leq k \leq 503$, $c_s(k) > 0$ and $c_v(k) = 0$, so this is clear.  For $357 \leq k \leq 487$,  we note that
in addition to the coalitions in $\cc_s(k)$ of type S1 above, there
are coalitions of type S2 or S3 and $\vp$ is never critical in these,  hence $c_k(s) >  c_v(k)$.

\end{proof}  

\vsp

\begin{prop}  \label{sr} If $c_s(k) = 0$, then $c_r(k) = 0$.  For all $k \in \nn$  such that $c_s(k) \neq 0$,  $c_s(k) > c_r(k)$.  
\end{prop}

\begin{proof}  The first statement follows immediately from Table 1. 
If $321 \leq k \leq 356$ or $392 \leq k \leq 503$, then $c_r(k) = 0$ and $c_s(k) \neq 0$, so there is nothing to prove.  We break the remaining values of $k$ into
three cases:  $270  \leq k \leq 320$, $k = 357$, and $358 \leq k \leq 391$. \\

\noindent
{\bf Case 1.}  For $270 \leq k \leq 320$,the coalitions in $\cc_s(k)$ are of type S1 and
thus  consist of $\sen$ plus $50$ others from the full Senate, $k - 52$
representatives and the president.  
Coalitions in $\cc_r(k)$ are of type R1 and hence consist of 
$\rep$ plus $217$ other representatives, the president, and $k - 219$ from the full Senate.   It follows that

$$c_s(k) = \binom{100}{50} \cdot \binom{435}{k - 52}, \ts \ts
c_r(k) =\binom{101}{k - 219} \cdot \binom{434}{217}.$$

We apply Proposition \ref{main} (d)  with $m_s = 101$, $m_r = 435$, $q_s = 51$, and $q_r = 218$.  The
conditions  $q_s m_r  > q_r m_s$ and $ (m_r-q_r)(q_s + 1) > (m_s -q_s)(q_r + 1)$  are easily checked.   Hence, by the proposition, $c_s(k) > c_r(k)$ for all $k$.  \vsp  

{\bf Case 2.}  $k = 357$.  Coalitions in $\cc_s(357)$ consist of $\sen$, $66$ other senators, and $290$ representatives, while
coalitions in $\cc_r(357)$ consist of $\rep$, $289$ other representatives, and $67$ senators.  Then
$$c_s(357) = \binom{99}{66} \cdot \binom{435}{290}  > \binom{100}{67} \cdot \binom{434}{289} = c_r(357).$$ \\

\noindent
{\bf Case 3.}  For $358 \leq k \leq 391$, the coalitions in $\cc_s(k)$ consist of $\sen$, $66$ other senators and either $k - 67$
representatives, or $\vp$ and $k - 68$ representatives.  The coalitions in $\cc_r(k)$ consist of  $\rep$, $289$ other
representatives, and either $k - 290$ senators or $\vp$ and $k - 291$ senators.  Thus

\begin{equation} \label{ineq2}
\begin{split} c_s(k) &= \binom{99}{66} \cdot \binom{435}{k - 67}  + \binom{99}{66} \cdot \binom{435}{k - 68}  \\
c_r(k) &= \binom{100}{k - 290} \cdot  \binom{434}{289} + \binom{100}{k - 291} \cdot \binom{434}{289}.    \end{split}  \end{equation}

By Proposition \ref{main}(d)   with $m_s = 100, m_r = 435, q_s = 67$, and $q_r = 290$, $$\binom{99}{66} \cdot \binom{435}{k - 67} >
\binom{100}{k - 290} \cdot \binom{434}{289},$$ and  with $m_s = 100, m_r = 435, q_s = 66$, and $q_r = 289$,
$$\binom{99}{66} \cdot \binom{435}{k - 68} >
\binom{100}{k - 291} \cdot \binom{434}{289}.$$  Therefore $c_s(k) > c_r(k)$.  

\end{proof}

\vsp

Finally, we compare the numbers $c_v(k)$ and $c_r(k)$.   Apart from a narrow range of $k$'s, $c_v(k)$ is the larger of the two.  \vsp

\begin{prop}  \label{rv}  (a)  If $c_v(k) = 0$, then $c_r(k) = 0$.

(b)  Suppose $k \in \nn$ such that $c_v(k) \neq 0$, so that $270 \leq k \leq 487$.  
If $270 \leq k \leq 356$ or $380 \leq k \leq 487$,  $c_v(k) > c_r(k)$.  For the remaining $k$, i.e., $357 \leq k \leq 379$,
$c_r(k) > c_v(k)$.
\end{prop}

\begin {proof}  (a)  This follows from Table 1.

(b)   Recall that coalitions in $C_v(k)$ consist of $\vp$, $50$ senators,
the president, and $k - 52$ representatives.  Thus, for all such values of $k$, we have
$c_v(k) = \binom{100}{50} \binom{435}{k - 52} $.    \vsp

Case 1:  For $270 \leq k \leq 356$, $c_v(k) = c_s(k) > c_r(k)$, by Proposition \ref{sv} and Proposition \ref{sr}.  

Case 2:  For $k = 357$, the only coalitions in which $\rep$ is critical are of type R2 and
consist of $\rep$ plus 289 other representatives and 67 senators.  Hence we have
$$c_v(357) = \binom{100}{50} \binom{435}{305} <  \binom{434}{289} \binom{100}{67} = c_r(357).$$  

Case 3:  For $358 \leq k \leq 390$,  coalitions in which $\rep$ is critical are of type R3  and they consist of $\rep$ plus
$289$ other representatives and $k - 290$ members of the full Senate.   Thus we have

\begin{equation} \label{rep}
c_r(k) = \binom{434}{289} \binom{101}{k - 290} 
\end{equation} for these $k$.    

Using the computer algebra software Mathematica, we find that $$\binom{434}{289} \binom{101}{k - 290}  >
\binom{100}{50} \binom{435}{k - 52}$$  iff $358 \leq k \leq 379$.

Case 4: For $k = 391$ we have $$c_v(391) = \binom{100}{50} \binom{435}{239} > \binom{434}{289} = c_r(391),$$   as claimed.

Case 5:  For $392 \leq k \leq 487$, $c_r(k) = 0$ and $c_v(k) > 0$, so $c_v(k) > c_r(k)$.   

\end{proof}

\vsp

\begin{theorem} In the simple game modeling the US legislative system, the weak desirability relation yields 
$\prs \succ \sen \succ \rep$ and $\sen \succeq \vp$.  It follows that

 \begin{enumerate} \item For any power index $\Phi$ for which $\Phi(p) \neq 0$ and $\Phi(s) \neq 0$, we have $$\prs >_\Phi \sen >_\Phi \rep
 \text{ and } \sen \geq_\Phi \vp.$$

\item If $\Phi$ is the  Banzhaf or  Shapley-Shubik index, we have $\prs >_\Phi \sen >_\Phi \vp >_\Phi \rep$.
\end{enumerate}

\end{theorem}

\begin{proof}  Propositions \ref{p}, \ref{sv}, and \ref{sr}  imply that $\prs \succ \sen \succ \rep$ and $\sen \succeq \vp$ and (a) follows from this by Proposition 1.   

(b):  By (a), we 
need only show that $\vp >_\Phi \rep$ if $\Phi$ is the Banzhaf or Shapley-Shubik index.      Recall that $c_v(k) \neq 0$ for $270 \leq k \leq 487$ and
for these $k$, $c_v(k) = \binom{100}{50} \binom{435}{k - 52}$, while $c_r(k) \neq 0$ for $270 \leq k \leq 391$.

Suppose $\Phi$ is Banzhaf power, then $$\Phi(\vp) = \frac{1}{2^{536}}  \sum_{k = 270}^{487} \binom{100}{50} \binom{435}{k - 52}.$$  For $\rep$ we must add up the contributions from the three types
of critical instances.  For type R1, $c_r(k) = \binom{434}{217} \binom{101}{k - 219}$, type R2  corresponds to $c_r(357) = \binom{434}{289} \binom{100}{67}$, and
for type R3, $c_r(k) = \binom{434}{289} \binom{101}{k - 290}$.   Thus 
$$\Phi(\rep)  = \frac{1}{2^{536}} \left( \sum_{k = 270}^{320} \binom{434}{217} \binom{101}{k - 219} +\binom{434}{289} \binom{100}{67} 
+ \sum_{270}^{487}  \binom{434}{289} \binom{101}{k - 290} \right).$$   It is easy to check that $\Phi(\vp) > \Phi(\rep)$ using Mathematica.   

If $\Phi$ is Shapley-Shubik  power then the calculation of $\Phi(\vp)$ and $\Phi(\rep)$ is as for Banzhaf power except that we must multiply each $c_r(k)$ and $c_v(k)$
by the weight $\la_k = 1/\left( n \binom{n-1}{k-1}\right)$.  Thus we need only use Mathematica to verify that
\begin{align*} & \sum_{k = 270}^{487}  \la_k \binom{100}{50} \binom{435}{k - 52}  > \\
&\sum_{k = 270}^{320} \la_k \binom{434}{217} \binom{101}{k - 219} +  \la_{357} \binom{434}{289} \binom{100}{67} 
+ \sum_{270}^{487} \la_k  \binom{434}{289} \binom{101}{k - 290}. \end{align*}
\end{proof}  

\vsp

\begin{remark}  The proofs of Propositions \ref{p}, \ref{sv},   and \ref{sr}  did not depend on the specific numbers of representatives 
and senators in the US system and the quotas in the sense that as long as the assumptions of Lemma \ref{lem1} and of
Proposition \ref{main} hold for $m_s, m_r, q_s$, and $q_r$, then the conclusions of these propositions hold.  
However, comparing the ranking of a representative and the vice-president using 
Proposition \ref{rv} involves the specific numbers in the US system and thus does not generalize immediately.   \vsp

\end{remark}

\subsection{Supermajority Rules} 

For some bills in the US Senate a supermajority of 60 or more
senators are required to vote yes in order to break a Filibuster.  In this case, we still have $\sen \succ \rep$ for $\sen$ a senator and
$\rep$ a representative  In fact, $S \succ R$ will hold regardless
of the number of votes needed to pass a bill in the Senate, as long as this number is greater than $50$.  \vsp

\begin{prop}   Suppose $q_s$ votes are needed to pass legislation in the Senate, where $51 \leq q_s \leq 100$, and
everything else remains the same.  Then we have $P \succ S \succ R$ as before.   
\end{prop}

\begin{proof}   The proof that $P \succ S$ generalizes immediately.   To show that $S \succ R$ we first note 
that  the conditions of Lemma \ref{lem1} are still satisfied and thus we have $c_s(k) = 0$ implies $c_r(0) = 0$.  
We can then apply  Proposition \ref{main} (d) from the appendix as in the proof of Proposition \ref{sr}.   The conditions needed are
(1) $q_s m_r  > q_r m_s$ and (2) $ (m_r-q_r)(q_s + 1) > (m_s -q_s)(q_r + 1) $.  In this case, we have $m_s = 100$, $m_r = 435$, and 
$m_r = 218$.  For (1) we have  $(a)(435) > (218)(100)$ iff $a \geq 51$, and for (2) we have
$(435 - 218)(a + 1) > (100 - a)(219)$, which holds iff $a \geq 50$.    This proves $S \succ R$.     \vsp

The proof that the $P \succ S$ does not depend on the number of senators needed to pass legislation when the
president votes yes, hence $P \succ S$ still holds in the this case.
\end{proof}   \vsp

Suppose that the House of Representatives decided to raise the quota for passing bills in order to insure that members
have more power than senators.
Assume that only $51$ votes are needed in the Senate.  Using Table 1, we
see that for winning coalitions containing
the president, the range of $k$ in which $c_s(k) \neq 0$ is $270 \leq k \leq 503$, thus in order to insure that there are 
no $k$ such that $c_s(k) \neq 0$ and $c_r(k) = 0$, there must be coalitions in which a representative is critical
with sizes up to $503$.  It follows that the minimum number of representatives 
needed to pass a bill in the House would have to be at least $401$ ($503 - 102$).   If $q_r = 401$ and $q_s$ remains $51$, then
$(401)(101) > (51)(435)$, so for $k = q_r + q_s + 1$ (minimal winning coalitions), by Lemma \ref{mincase}, $c_r(k) > c_s(k)$.
For the largest possible coalitions for which $c_s(k) \neq 0$,  i.e., $k = 503$, we have
$$c_r(503) = \binom{434}{400} \binom{101}{101}  > \binom{100}{50} \binom{435}{435} = c_s(503).$$
Since $c_r(k) > c_s(k)$ for $k$ the smallest and largest possible values, by Corollary \ref{cor2}, we have $R \succ S$
in this case.   Hence it is possible for the House to adopt a supermajority rule which would give their members
more power than members of the Senate.  

\vsp

\begin{remark}  The above discussion shows that in a bicameral legislative system members of a chamber can
increase their power by adopting supermajority a higher quota for passage of bills.
\end{remark}

\vsp
\section{Concluding Remarks}

We have investigated rankings induced by power indices of players and the weakly desirable relation in simple game models of 
bicameral legislative
systems.   In three cases (depending on the relative parity of the sizes of the houses), a member of the smaller house is ranked above a member of the larger house, regardless of
the power index used.   We showed that a sufficient condition for the members of the smaller house to always have more power than
the members of the larger house is that $q_s/m_s > q_r/m_r$ , where $m_s < m_r$ are the sizes of the houses and $q_r, q_s$ are the thresholds needed
to pass a bill, i.e., 
 the proportion of the smaller house needed to pass a bill is greater than the
proportion of the larger house needed to pass a bill.

In the fourth (exceptional case), where the size $m_s$ of the smaller house is odd and the size $m_r$ of the bigger house is even, if
$m_r \leq 2m_s$, then the condition $q_s/m_s > q_r/m_r$ fails.    When the gap is as small as possible, i.e., $m_r = m_s + 1$, then members of the larger house will
be ranked above members of the smaller house by all power indices.  For the largest possible gap, i.e., $m_r = 2m_s$, 
the larger house has an advantage only for minimal winning coalitions.   All of these results generalize to multicameral
legislatures.

Our main application of these results is to a standard simple game model of the US legislative system.   We showed that the
president always has the most power, a senator always has more power than a representative, and that a representative
has more power than the vice-president for most power indices, including the Banzhaf and Shapley-Shubik indices.   Most
of these results apply to similar legislative systems.  

\section{Appendix:  Products of Binomial Coefficients}

We prove the technical results on products of binomial coefficients that are needed to compare the numbers  $c_i(k)$ for different
players in our simple game models.  
Here are some basic facts about
binomial coefficients: 

$$\binom{n}{k} = \frac{n!}{k! (n - k)!}, \quad  \binom{n}{k} + \binom{n}{k+1} = \binom{n+1}{k+1}, \quad \binom{n}{k+1}/ \binom{n}{k} = \frac{n - k}{k+1} 
$$

Recall that in order to prove that $c_s(k) > c_r(k)$, we must prove that inequality \eqref{ineq1} holds for $k$.   For the convenience of the reader, 
the inequality is given again:

$$\binom{m_s - 1}{q_s - 1} \cdot \binom{m_r}{k - q_s} >   \binom{m_r-1}{q_r-1} \cdot  \binom{m_s}{k - q_r}  \quad \quad \quad \quad \quad (1) $$ \vsp

This holds iff
 
\begin{equation} \label{ineq2}  \binom{m_r}{k - q_s}/\binom{m_r - 1}{q_r -1}  > \binom{m_s}{k - q_r} / \binom{m_s - 1}{q_s - 1} .  
\end{equation} 

\vsp

Define the following
functions:  For positive integers $u < p$ and an integer $i$ such that $0 \leq i < p - u$, define 
$$f(p, u, i) := \binom{p}{u + i}/\binom{p-1}{u - 1},$$
$$g(p,u,i) := \frac{ p - u - i}{u + i + 1}.$$   

Then inequality \eqref{ineq1} holds iff inequality \eqref{ineq2} holds iff $$f(m_r,q_r,k - q_r - q_s) > f(m_s,q_s, k - q_r - q_s).$$  \vsp

\begin{lemma}\label{fg}  With $f$ and $g$ as above, 
 $$f(p,u,i+1) = g(p,u,i) \cdot f(p,u, i)$$    for all $i$  with $0 \leq i < p - u$,\end{lemma}

\begin{proof}
Using basic properties of  binomial coefficients, we have
$$f(p,u, i+ 1)/f(p,u,i)  =  \binom{p}{u + i + 1}/ \binom{p}{u+i} =  \frac{p  -u -  i}{u + i + 1},$$
which proves the claim.  
 
  \end{proof}    \vsp 
  
  \begin{lemma} \label{mincase}  Let $m_r, m_s, q_r, q_s \in \nn$ such that $1 < q_r < m_r$ and $1 < q_s < m_s$.  Then
   \eqref{ineq1} holds for $k = q_r + q_s$ if and only if  $q_s m_r  > q_r m_s$.  If $q_s m_r = q_r m_s$, then \eqref{ineq1} holds with
   $>$ replaced by $=$.  
   \end{lemma}
   
   \begin{proof}    Inequality \eqref{ineq1} holds for $k =  q_r + q_s$ iff $f(m_r,q_r,0) > f(m_s,q_s,0)$.  It is easy to check that
   $f(p,u,0) = p/u$, hence \eqref{ineq1}  holds iff $m_r/q_r > m_s/q_s$ iff $q_s m_r  > q_r m_s$.  If $q_s m_r = q_r m_s$, then
   $f(m_r,q_r,0) = f(m_s,q_s,0)$ and we have equality in \eqref{ineq1}.
   \end{proof}

  \begin{prop} \label{main}   Let  $m_r, m_s, q_r, q_s, \in \nn$ such that
 $m_s < m_r$, $1 < q_s < m_s$, and $1 < q_r < m_r$.  Let $N = \min\{m_s + q_s, m_r +  q_r \}$. Then
 
 \begin{enumerate}

 \item Inequality \eqref{ineq1} holds for $k = q_r + q_s + 1$ if and only if  $$\frac{(m_r - q_r)}{(q_r + 1)} \frac{m_r}{q_r}  > \frac{(m_s - q_s)}{(q_s + 1)} \frac{m_s}{q_s}.$$
 
  \item  Suppose $d \in \nn$ with  $0  \leq d \leq N$.   If $g(m_r,q_r,d) > g(m_s,q_s,d)$, then $g(m_r,q_r,i) > g(m_s,q_s,i)$ for 
  all $i$ such that $d \leq i \leq N$. 
  
    \item  Suppose $d \in \nn$ with  $0  \leq d \leq N$.   If $f(m_r,q_r,d) > f(m_s,q_s,d)$ and $g(m_r,q_r,d) > g(m_s,q_s,d)$,
 then $f(m_r,q_r,i) > f(m_s,q_s,i )$ for 
  all $i$ such that $d \leq i \leq N$. 
 
\item     If $f(m_r,q_r,0) > f(m_s,q_s,0)$ and $g(m_r,q_r,0) > g(m_s,q_s,0)$, then inequality \eqref{ineq1} holds for all $k$ such that $q_r + q_s  \leq k \leq  N$.
It follows that  inequality \eqref{ineq1} holds for all $k$ such that $q_r + q_s  \leq k \leq  N$ if the following
 two inequalities hold:
 $$q_s m_r  > q_r m_s,  \quad (m_r-q_r)(q_s + 1) > (m_s -q_s)(q_r + 1).$$

 \end{enumerate}
 
 \end{prop}
 
 \begin{proof}   As noted above, inequality \eqref{ineq1} holds for $k$ iff $f(m_r,q_r, k - q_r - q_s) > f(m_s,q_s, k - q_r - q_s)$.  \vsp

(a) Inequality \eqref{ineq1} holds for $k = q_r + q_s + 1$ iff $f(m_r,q_r,1) > f(m_s,q_s,1)$.  By Lemma \ref{fg}, $f(p,u,1) = g(p,u,0) \cdot f(p,u,0)$, hence we need 
$$g(m_r,q_r,0) \cdot f(m_r,q_r,0) > g(m_s,q_s,0) \cdot f(m_s,q_s,0), $$ which yields the claimed inequality.  \vsp

(b)  The proof is by induction on $i$.  By definition,  $g(m_r,q_r,i) > g(m_s,q_s,i)$ is $$\frac{m_r - q_r - i}{q_r + i + 1} > \frac{m_s - q_s - i}{q_s + i + 1}$$ which is equivalent to  $m_r(q_s + i) + m_r - q_r > m_s(q_r + i) + m_s - q_s$.  
 By assumption this holds for $i =d$.  Assume it is true for some $i$ with $d \leq i < N$.
 Then
$$m_r(q_s + i + 1) + m_r - q_r  = m_r(q_s + i) + ( m_r - q_r ) + m_r  > $$  $$m_s(q_r + i) + m_s - q_s + m_s = 
m_s(q_r + i + 1) + m_s  - q_s,$$ since 
$m_r > m_s$.
It follows that $g(m_r,q_r,i+1) > g(m_s,q_s,i+1)$ and we are done by induction.  \vsp

(c)     Since $f(p,u,i+1) = g(p,u,i) \cdot f(p,u,i)$ by Lemma \ref{fg}, 
the claimed result follows easily from (b) by induction.    \vsp

 (d) This follows from (c) with $d = 0$, noting that $g(m_r,q_r,0) > g(m_s,q_s,0)$ iff the second inequality holds.

 \end{proof}

  \begin{corollary} \label{cor2}   Let  $m_r, m_s, q_r, q_s, \in \nn$ such that
 $m_s < m_r$, $1 < q_s < m_s$, $1 < q_r < m_r$ and let $N = \min\{ m_s-q_s, m_r - q_r\}$.   Suppose we have
  $a$ such that $q_r + q_s + 1   \leq a \leq N$ and  \eqref{ineq1} holds with $k = a$.
 Then 
 \eqref{ineq1} holds for all $k $ such that $a \leq k \leq N$.
 \end{corollary}

 \begin{proof}  It is enough to assume that $a$ is minimal such that \eqref{ineq1} holds for $k = a$. For ease of
 exposition, let $i = a - q_r - q_s$.   Then, by minimality of $a$,
  we have $f(m_r,q_r, i - 1) \leq f(m_s,q_s, i - 1)$  
 and $f(m_r,q_r, i ) > f(m_s,q_s,i )$.  By Lemma \ref{fg}, $ g(m_r,q_r, i - 1) \cdot 
 f(m_r,q_r, i - 1) > 
 g(m_s,q_s, i - 1) \cdot f(m_s,q_s,i - 1)$.  The two inequalities imply $g(m_r,q_r,i - 1) > 
 g(m_s,q_s,i - 1)$.
 Then, by Proposition \ref{main}(b), $g(m_r,q_r,i) > g(m_s,q_s,i)$. The result now
 follows from Proposition \ref{main} (c).  
 \end{proof}

 \bibliographystyle{amsplain}

\begin{thebibliography}{1}

\bibitem{banz}
J.F. Banzhaf, \emph{Weighted voting doesn't work: a mathematical analysis},
  Rutgers Law Review \textbf{19} (1965), 317--343.

\bibitem{bak}
S.~Brams, P.~Affuso, and D.M. Kilgour, \emph{Presidential power: a
  game-theoretic analysis}, The Presidency in American Politics (P.~Brace,
  C.~Harrington, and G.~King, eds.), New York University Press, 1989,
  pp.~55--74.

\bibitem{CF}
F.~Carreras and J.~Freixas, \emph{On ordinal equivalence of power measures
  given by regular semivalues}, Math. Social Sci. \textbf{55} (2008), 221--234.

\bibitem{dub}
P.~Dubey, P.~Neyman, and R.J. Weber, \emph{Value theory without efficiency},
  Math. Oper. Res. \textbf{6} (1981), 122--128.

\bibitem{SK}
D.~Saari and K.~Sieberg, \emph{Some surprising properties of power indices},
  Games and Economic Behavior \textbf{36} (2000), 241--263.

\bibitem{SS}
L.S. Shapley and M.~Shubik, \emph{A method for evaluating the distribution of
  power in a committee system}, American Political Science Review \textbf{48}
  (1954), 787--792.

\bibitem{tp}
A.~Taylor and A.~Pacelli, \emph{Mathematics and politics}, second ed.,
  Springer, New York, 2008, Strategy, voting, power and proof.

\bibitem{web}
R.~J. Weber, \emph{Subjectivity in the valuation of games}, Game Theory and
  Related Topics, North Holland, 1979, pp.~129--136.

\end{thebibliography}
\providecommand{\bysame}{\leavevmode\hbox to3em{\hrulefill}\thinspace}
\providecommand{\MR}{\relax\ifhmode\unskip\space\fi MR }
\providecommand{\MRhref}[2]{%
  \href{http://www.ams.org/mathscinet-getitem?mr=#1}{#2}
}
\providecommand{\href}[2]{#2}

\end{document}